\documentclass[12pt]{article}
\usepackage[utf8]{inputenc}
\usepackage[numbers]{natbib}

\usepackage{amsmath, amsfonts, amssymb, amsthm, graphicx}
\usepackage{authblk}
\usepackage{verbatim}

\usepackage[a4paper, left=1.1in,right=1.1in]{geometry}
\usepackage{setspace}
\usepackage{xcolor}
\definecolor{myblue}{RGB}{0, 0, 139}
\usepackage[colorlinks, citecolor=myblue, linkcolor=myblue]{hyperref}

\newtheorem{proposition}{Proposition}
\newtheorem{theorem}{Theorem} 
\newtheorem{lemma}{Lemma}

\newtheorem{corollary}{Corollary}

\newcommand{\norm}[1]{\left\lVert#1\right\rVert}
\newcommand{\W}{\mathcal{W}}
\newcommand{\C}{\mathcal{C}}
\newcommand{\e}{\epsilon}
\newcommand{\w}{\omega}
\renewcommand{\l}{\ell}

\renewcommand{\P}{\mathbb{P}}

\newcommand{\R}{\mathbb{R}}

\let\temp\phi
\let\phi\varphi
\let\varphi\temp

\title{Exact Convergence Analysis for Metropolis-Hastings Independence Samplers in Wasserstein Distances}

\author{Austin Brown\thanks{brow5079@umn.edu} }    
\author{Galin L. Jones\thanks{galin@umn.edu}}    

\affil{School of Statistics, University of Minnesota, Minneapolis, MN, USA}

\begin{document}

\maketitle

\begin{abstract}
Under mild assumptions, we show the exact convergence rate in total variation is also exact in weaker Wasserstein distances for the Metropolis-Hastings independence sampler. 
We develop a new upper and lower bound on the worst-case Wasserstein distance when initialized from points.
For an arbitrary point initialization, we show the convergence rate is the same and matches the convergence rate in total variation.
We derive exact convergence expressions for more general Wasserstein distances when initialization is at a specific point. 

Using optimization, we construct a novel centered independent proposal to develop exact convergence rates in Bayesian quantile regression and many generalized linear model settings.
We show the exact convergence rate can be upper bounded in Bayesian binary response regression (e.g. logistic and probit) when the sample size and dimension grow together.
\end{abstract}

\section{Introduction}
\label{section:1}

Applications in modern statistics often require generating Monte Carlo samples from a distribution defined on $\Theta \subseteq \R^d$ using a version of the Metropolis-Hastings algorithm \citep{Brooks2011, Hastings1970, Metropolis1953}. 
Popular versions of Metropolis-Hastings include, among many others, random walk
Metropolis-Hastings, Metropolis-adjusted Langevin algorithm,
Metropolis-Hastings independence (MHI) sampler, and Hamiltonian Monte Carlo.

Convergence analyses of general state space Metropolis-Hastings algorithms have traditionally focused on studying their convergence rates in total variation distances \citep{Meyn2009, Rosenthal1995, Tierney1994}. These convergence rates have received significant attention, at least in part, because they provide a key sufficient condition for the existence of central limit theorems \citep{Jones2004} and the validity of methods for assessing the reliability of the simulation effort \citep{robe:etal:viz:2021, Vats2019}. However, convergence analyses of Metropolis-Hastings Markov chains typically result in {\em qualitative} convergence rates \citep{Hairer2014, Jarner2000, Johnson2012, Mengersen1996, Roberts1996geo}.  Even quantitative upper bounds on the convergence rate have been rare with the exception of some MHI samplers \citep{Tierney1994}. 
Exact convergence rates in total variation in the general state space setting have been nonexistent until very recently, and, again, these are for MHI samplers \citep{Smith1996, Wang2022}. 
However, exact convergence rates in total variation have been studied only in trivial examples.
This leaves practitioners with little guidance on the convergence behavior and reliability of MHI samplers in practical applications, especially in high dimensions.
%However, exact convergence rates for MHI Markov chains often depend on unknown quantities and this leaves practitioners with little guidance on how to proceed in specific applications.
%Developing such convergence analyses for Metropolis-Hastings Markov chains is an apparently difficult problem and this leaves practitioners with little guidance on how to proceed in specific applications.

At the same time, there has been significant recent interest in the convergence properties of Monte Carlo Markov chains in high-dimensional settings \citep{Durmus2019, KarlOskar2019, Hairer2014, john:etal:2019, qin2019convergence, raja:spar:2015, yang:etal:2016} and traditional approaches can have limitations in this regime \citep{qin2021limitations}. 
This has led to an interest in considering the convergence rates of Monte Carlo Markov chains using Wasserstein distances 
\citep{Gibbs2004, Hairer2014, jin:tan:2020, Madras2010, qin2021bounds, qin2021wasserstein} which may scale to large problem sizes where other approaches have had difficulties \citep{Durmus2015, Hairer2014, qin2021wasserstein}.
Convergence analyses in Wasserstein distances also result in benefits similar to those obtained using total variation such as central limit theorems and concentration inequalities for time averages of the Markov chain \citep{Hairer2014, jin:tan:2020, Joulin2010, Komorowski2011}. 

We study exact convergence rates of the MHI sampler in $L_1$-Wasserstein distances, which we refer to as only Wasserstein distances.
%, alternative to total variation.
The exact convergence behavior of Metropolis-Hastings algorithms across various Wasserstein distances has not been previously studied.
Only in specific Wasserstein distances, have upper bounds been developed \citep{Durmus2015, Hairer2014}.
We develop exact convergence rates which are universal across Wasserstein distances for the MHI sampler. 
Surprisingly, we show the exact convergence rate in total variation \citep{Wang2022} is also exact for Wasserstein distances weaker than total variation under mild assumptions.
We develop a new upper and lower bound on the worst-case Wasserstein distance when initialized from points.
Under mild assumptions similar to the ones used for the result in total variation \citep{Wang2022}, we show the convergence rate at any point initialization is the same as the worst-case convergence rate.
When the algorithm is started at a specific point, we give exact convergence expressions across more general Wasserstein distances.

Our theoretical results on the exact convergence rate extend the results in total variation \citep{Wang2022} to Wasserstein distances.
However, only a trivial example was studied in total variation \citep{Wang2022}.
We provide a practically relevant application of our theoretical results by developing exact convergence expressions using normal-inverse-gamma proposals in the Bayesian quantile regression setting. Previously, qualitative convergence results for a Gibbs sampler were developed \citep{Khare2012}.
%Due to its simplicity, MHI samplers can be computationally efficient at each iteration in contrast to more sophisticated Markov chain Monte Carlo algorithms.
%Compared to methods used to approximate integrals such as importance sampling, MHI can be a useful alternative for practitioners as it generates samples from the target distribution. 

Compared to methods used to approximate integrals such as importance sampling, MHI samplers can generate samples from the target distribution which is often of interest for practitioners. 
MHI samplers can also be computationally efficient at each iteration in contrast to more sophisticated Markov chain Monte Carlo algorithms but tend to require many iterations to accept a proposed sample.
%MHI samplers tend to require many iterations to accept proposed samples but can also be computationally efficient at each iteration in contrast to more sophisticated Markov chain Monte Carlo algorithms.
Connections between the MHI sampler and rejection sampling are also well-known \citep{Liu1996, Tierney1994}. Exact convergence rates for MHI samplers may also provide further insight into more popular Metropolis-Hastings algorithms such as the MALA and RWM algorithms.

Motivated by the general theoretical work we consider using a centered Gaussian proposal and derive exact convergence expressions in Wasserstein distances for a large class of target distributions. The centered Gaussian proposal matches the maximal point of the proposal density with that of the target density.
By centering an independent proposal, we directly imbue the Markov chain with a strong attraction to a set where the target distribution has high probability.
This centered Gaussian proposal is similar to using a Laplace approximation \citep{Pierre2011, Shephard1997, Tierney1994}, but differs in its covariance matrix. 
We study this MHI in several Bayesian generalized linear models and derive exact convergence expressions in general Wasserstein distances.

Our techniques are based on a condition \citep{Mengersen1996, Tierney1994, Wang2022} which is well-known but has previously been difficult to scale in high-dimensional settings.
The novelty in our analysis is a carefully constructed proposal to develop exact convergence rates across Wasserstein distances.
We then consider scaling properties of the exact convergence rate to large dimensions and sample sizes in high-dimensional Bayesian binary response regression (e.g. logistic and probit regression) with Gaussian priors.
Data augmentation algorithms have been developed for these models \citep{Chib1993, Polson2013}, but the required matrix inversions at each iteration can be computationally intensive.
We derive an explicit asymptotic upper bound on the convergence rate of our MHI for general Wasserstein distances when the sample size and dimension increase in such a way that the ratio $d/n \to \gamma \in (0, +\infty)$.
In this case, 
we show 
%a well-known condition yields 
informative convergence rates for practitioners for the MHI sampler which can scale to large problem sizes when the convergence analysis is exact.

To the best of our knowledge, this work is the first to successfully address the convergence complexity of Metropolis-Hastings in general Wasserstein distances when both the sample size and the dimension increase. 
Previously under the conditions of a central limit theorem, the convergence complexity in total variation of Random Walk Metropolis (RWM) on a compact set was studied \citep{Belloni2009}.
In contrast, our convergence complexity results do not rely on the underlying space being compact.
The dimension dependence of the mixing time has been studied in specific Wasserstein distances and total variation for Metropolis-Hastings algorithms such as Metropolis-adjusted Langevin (MALA) and RWM for certain log-concave target distributions \citep{Dwivedi2018, Eberle2014}.
We take into account the sample size and upper bound the convergence rate which provides further theoretical guarantees for time averages of the Markov chain \citep{Jones2004, Joulin2010}.
Related results have investigated the convergence properties of some high-dimensional Gibbs samplers \citep{Papaspiliopoulos2019, Papaspiliopoulos2021} or studied in the cases when the dimension or the sample size increase individually \citep{KarlOskar2019, john:etal:2019, qin2021wasserstein, raja:spar:2015}. 

The remainder is organized as follows.  In Section~\ref{section:prelim}, we
define the Metropolis-Hastings independence sampler and the Wasserstein distance.
In Section~\ref{sec:MHIS}, we develop exact convergence rates in the Wasserstein distance for the MHI sampler and apply this theory to Bayesian quantile regression. 
In Section~\ref{section:4}, we study a centered Gaussian
proposal to obtain exact convergence expressions and apply this to many popular Bayesian generalized linear models used in statistics. We also develop high-dimensional convergence complexity results for Bayesian binary response regression in the large dimension and large sample size regime. Section~\ref{section:conclusion} contains some final remarks.  Some technical details and proofs are deferred to the appendices.

\section{MHI samplers and Wasserstein distances}\label{section:prelim}

As they will be considered here, MHI samplers
simulate a Markov chain with invariant distribution $\Pi$ supported on a
nonempty set $\Theta \subseteq \R^d$ using a proposal distribution
$Q$ which, to avoid trivialities, is assumed throughout to be
different than $\Pi$. 
We also assume throughout that $\Pi$ has Lebesgue density $\pi$ with support $\Theta$ and $Q$ has Lebesgue density $q$ with support $\Theta$. 
Define
\[
a(\theta, \theta') = 
\begin{cases}
\min\left\{ \frac{ \pi(\theta') q(\theta) }{\pi(\theta)
    q(\theta')}, 1 \right\}, &\text{ if }\pi(\theta) q(\theta') > 0 \\ 
1, & \text{ if } \pi(\theta) q(\theta') = 0
\end{cases}.
\]
We will consider MHI samplers initialized at a point $\theta_0 \in \Theta$. MHI proceeds as follows: for $t \in \{1, 2, \ldots \}$, given $\theta_{t-1}$, draw $\theta'_t \sim Q(\cdot)$ and $U_t \sim \text{Unif}(0, 1)$ independently so that 
\[
\theta_{t} = 
  \begin{cases}
    \theta_t', \hspace{.5cm}\text{if } U_t 
    \le a\left( \theta_{t-1}, \theta'_t \right) 
    \\ 
    \theta_{t-1}, \hspace{.15cm}\text{otherwise}
  \end{cases}.
\]
If $\delta_\theta$ denotes the Dirac measure at the point $\theta$,
the MHI Markov kernel $P$ is defined for
$\theta \in \R^d$ and $B \subseteq \R^d$ by
\[
P(\theta, B)
= \int_{B} a(\theta, \theta') q(\theta') d\theta'
+ \delta_\theta(B) \left( 1 - \int a(\theta, \theta') q(\theta') d\theta' \right) .
\]
For $\theta \in \R^d$, define the Markov kernel at iteration time $t \ge 2$ recursively by
\[
P^t(\theta, B)
= \int P(\theta, d\theta') P^{t-1}(\theta', B).
\]

Let $\C(P^t(\theta, \cdot), \Pi)$ be the set of all joint probability
measures with marginals $P^t(\theta, \cdot)$
and $\Pi$ and $\rho$ be a lower semicontinuous metric. The $L_1$-Wasserstein distance \citep{Kantorovich1957, Villani2003, Villani2008}, which we will call simply the Wasserstein distance, is
\[
  \W_{\rho}\left( P^t(\theta, \cdot), \Pi \right) = \inf_{\xi \in
    \C(P^t(\theta, \cdot), \Pi)} \int
  \rho(\theta, \w) d\xi(\theta, \w) .
\]
Notice that when the metric $\rho$ is  $\rho(\theta, \w) = I_{\theta \not= \w}$, then the Wasserstein distance is the total variation distance $\norm{\cdot}_{\text{TV}}$.
More generally for a lower semicontinuous function $V \ge 1$, $\rho(\theta, \w) = [V(\theta) + V(\w)] I_{\theta \not= \w}$ defines a weighted total variation distance.
Another example is $\rho(\theta, \w) = \min\{ \norm{ \theta - \w },  1 \}$ which is always less than the Hamming metric used in total variation.
More general $L_p$-Wasserstein distances with $p \ge 2$ are not studied in this work. 

\section{Exact convergence rates for MHI samplers}
\label{sec:MHIS}

When the ratio of the proposal and target densities is bounded below by a positive number, that is, 
\[
\e^* = \inf_{\theta \in \Theta} \{ q(\theta)/\pi(\theta) \} > 0,
\] 
the MHI sampler is uniformly ergodic in total variation with convergence rate upper bounded by $1 - \e^*$ \citep[Corollary 4]{Tierney1994}. 
Unlike in accept-reject sampling, $\e^*$ does not need to be known explicitly or computed in order to implement MHI.
However, this requirement was shown to be necessary for uniform ergodicity in total variation \citep[Theorem 2.1]{Mengersen1996}. More recently, it was shown the convergence rate cannot be improved \citep[Theorem 1]{Wang2022}. 
%We show this is the case even in weaker Wasserstein distances.
We show this is the case even in weaker Wasserstein distances where the lower bound does not follow trivially from that of the total variation lower bound \citep[Theorem 1]{Wang2022}. 

\begin{theorem}
\label{thm:sharp_rate}
Suppose $\rho(\cdot, \cdot) \le 1$.
Then
\begin{align*}
\sup_{\theta \in \Theta} \W_\rho(P^t(\theta, \cdot), \Pi)
\le (1 - \e^*)^t \sup_{\theta \in \Theta} \int \rho(\cdot, \theta) d\Pi.
\end{align*}
If in addition, $q$ is lower semicontinuous on $\Theta$, $\pi$ is upper semicontinuous on $\Theta$, and
$\Theta$ can be expressed as a countable union of compact sets, then
\begin{align*}
(1 - \e^*)^t \inf_{\theta \in \Theta} \int \rho(\cdot, \theta) d\Pi
\le 
\sup_{\theta \in \Theta} \W_\rho(P^t(\theta, \cdot), \Pi)
\le (1 - \e^*)^t \sup_{\theta \in \Theta} \int \rho(\cdot, \theta) d\Pi.
\end{align*}
\end{theorem}

\begin{proof}
The proof is provided in Appendix~\ref{proof:thm:sharp_rate}.
\end{proof}

The semicontinuity assumption is not required when working with the total variation distance \citep[Theorem 1]{Wang2022}, but it is a mild assumption that holds in many practical applications.
The upper bound constant can improve upon upper bounds in total variation \citep{Tierney1994} if for example, $\rho$ is continuous and $\Theta$ is compact.
If $\e^* = 0$, Theorem~\ref{thm:sharp_rate} also gives the lower bound
\[
\inf_{\theta \in \Theta} \int \rho(\cdot, \theta) d\Pi
\le 
\sup_{\theta \in \Theta} \W_\rho(P^t(\theta, \cdot), \Pi),
\]
which shows MHI cannot converge uniformly from any starting point for many Wasserstein distances.
Thus, under mild assumptions, Theorem~\ref{thm:sharp_rate} gives a complete characterization of the worst-case convergence of the MHI sampler in many Wasserstein distances.

Exact convergence expressions are available when the Markov chain is initialized at $\theta^* = \text{argmin}\left\{ q(\theta) / \pi(\theta) : \theta \in \Theta \right\}$ using techniques from \citep{Wang2022}.  

\begin{proposition}
\label{prop:exactwasserstein}
Suppose  there exists a solution 
\[
\theta^* = \text{argmin}\left\{ q(\theta) / \pi(\theta) : \theta \in \Theta \right\}.
\] 
Then
\[
\W_{\rho}\left( P^t(\theta^*, \cdot), \Pi \right)
= \left( 1 - q(\theta^*)/\pi(\theta^*) \right)^t \int \rho(\theta, \theta^*) d\Pi(\theta) .
\]
\end{proposition}

\begin{proof}
Define $\e_{\theta^*} = q(\theta^*)/\pi(\theta^*) = \e^*$.
Under our assumptions $P^t(\theta^*, \cdot)$ can be represented as a convex combination of the target distribution and the Dirac measure at the point $\theta^*$
\citep[][Remark 1, Theorem 2]{Wang2022}, that is,
\begin{align}
P^t(\theta^*, \cdot)
&= \left( 1 - \left( 1 - \e_{\theta^*} \right)^t \right) \Pi + \left( 1 - \e_{\theta^*} \right)^t \delta_{\theta^*} .
\end{align}
Let $\psi : \Theta \to \R$ be a function such that $\int_\Theta |\psi| d\Pi < \infty$.
We have
the identity,  
\begin{align}
\int_\Theta \psi dP^t(\theta^*, \cdot)
= \left( 1 - \left( 1 - \e_{\theta^*} \right)^t \right) \int_\Theta \psi
  d\Pi(B) + \left( 1 - \e_{\theta^*} \right)^t
  \psi(\theta^*). \label{eq:exactbias} 
\end{align}
Since the only coupling between $\Pi^*$ and the Dirac measure $\delta_{\theta^*}$ is the independent coupling \citep{Giraudo}, the Wasserstein distance takes the simple form $\W_{\rho}\left( \delta_{\theta^*}, \Pi \right) = \int \rho(\theta, \theta^*) d\Pi(\theta)$. 

Since $q$ is not exactly $\pi$, then $\e_{\theta^*} \in (0, 1)$.
Let $M_b(\R^d)$ be the set of bounded measurable functions on $\R^d$ and for real-valued functions $\phi$, let $\norm{\phi}_{\text{Lip}(\rho)} = \sup_{x, y, x \not= y} \{ |\phi(x) - \phi(y)| / \rho(x, y) \}$ denote the Lipschitz norm with respect to the distance $\rho$.
Applying the Kantorovich-Rubinstein theorem \citep[][Theorem 1.14]{Villani2003},
\begin{align*}
\W_{\rho}\left( P^t(\theta^*, \cdot), \Pi \right)
&= \sup_{\substack{\phi \in M_b(\R^d) \\ \norm{\phi}_{\text{Lip}(\rho)} \le 1}} \int_{\Theta} \phi d \left( P^t(\theta^*, \cdot) - \Pi \right)
\\
&= \sup_{\substack{\phi \in M_b(\R^d) \\ \norm{\phi}_{\text{Lip}(\rho)} \le 1}} \left\{ (1 - \e_{\theta^*})^t \int_{\Theta} \phi d \left( \delta_{\theta^*} - \Pi \right) \right\}
\\
&=  (1 - \e_{\theta^*})^t \sup_{\substack{\phi \in M_b(\R^d) \\ \norm{\phi}_{\text{Lip}(\rho)} \le 1}}\int_{\Theta} \phi d \left( \delta_{\theta^*} - \Pi \right)
\\
&= 
(1 - \e_{\theta^*})^t \W_{\rho}\left( \delta_{\theta^*}, \Pi \right)
\\
&= (1 - \e_{\theta^*})^t \int_{\Theta} \rho(\theta, \theta^*) d\Pi(\theta).
\end{align*}

\end{proof}.

\subsection{Application: Bayesian quantile regression}

Fix $r \in (0, 1)$ and suppose, for $i=1,\ldots,n$, $\e_i$ are independent and identically distributed (i.i.d.) with density 
\[
p_r(\e) = r(1 - r) \left( \exp((1 - r) \e ) I_{\e < 0} + \exp(-r\e) I_{\e \ge 0} \right).
\]
Let $v_0, s_0 \in (0, \infty)$ and $C \in \R^{d \times d}$ be symmetric positive-definite. 
We paramaterize the inverse gamma distribution so that if $\sigma \sim \text{IG}(v, s)$ for some $v, c \in (0, \infty)$, then it has a density proportional to $\sigma^{-v - 1} \exp(-s/\sigma)$.
Assume the Bayesian quantile regression model for $i \in 1, \ldots, n$ where $X_i \in \R^d$ is fixed and
\begin{align*}
&\sigma \sim \text{IG}(v_0, s_0)
\\
&\beta | \sigma \sim \text{N}_d(0, \sigma C)
\\
&Y_i = \beta^T X_i + \sigma \e_i.
\end{align*}
Let $\Pi(\cdot | X, Y)$ denote the posterior and $\pi(\cdot | X, Y)$ denote the density for this Bayesian model with normalizing constant $Z_{\Pi(\cdot | X, Y)}$.

Upper bounds on the convergence rate were previously investigated for Gibbs samplers \citep{Khare2012} in this setting.  We will study the MHI sampler with a normal-inverse-gamma proposal constructed as follows. Define $\l_r(u) = u (r - I_{u < 0})$ and $s_{n, r}: \R^d \to \R$ by $s_{n, r}(\beta) = \sum_{i = 1}^n \l_r(Y_i - \beta^T X_i) + \beta^T C^{-1} \beta / 2$. Since $s_{n, r}$ is strongly convex, let $\beta^* \in \R^d$ be the unique minimum of the function $s_{n, r}$. Now the MHI proposal is given by
\begin{align*}
&\sigma \sim \text{IG}(n + v_0, s_0 + s_{n, r}(\beta^*))
\\
&\beta | \sigma \sim \text{N}_d(\beta^*, \sigma C).
\end{align*}
Let $\Gamma : \R \to \R$ be the usual Gamma function and define 
\[
\e_{\beta^*} = Z_{\Pi(\cdot | X, Y)} \left( 2 \pi \right)^{-\frac{d}{2}} \det\left( C \right)^{-1/2} \left( s_0 + s_{n, r}(\beta^*) \right)^{n + v_0} \Gamma(n + v_0)^{-1}.
\]
The following gives an exact convergence rate of this algorithm which completely characterizes its convergence from a specific initialization.

\begin{theorem}
\label{theorem:quantilereg}
For any $\sigma_0 \in (0, \infty)$
\[
\W_{\rho}\left( P^t( (\beta^*, \sigma_0), \cdot ), \Pi(\cdot | X, Y) \right)
= \left( 1 - \e_{\beta^*} \right)^t \int \rho((\beta, \sigma), (\beta^*, \sigma_0)) d\Pi(\beta, \sigma | X, Y).
\]
\end{theorem}

\begin{proof}
We may define the function $f : \R^d \times (0, \infty) \to \R$ by
\begin{align*}
f(\beta, \sigma)
= \frac{s_0 + s_{n, r}(\beta)}{\sigma} + (n + v_0 + 1 + d/2) \log(\sigma)
\end{align*}
and write the posterior density as $\pi(\beta, \sigma | X, Y) = Z_{\Pi(\cdot | X, Y)}^{-1} \exp\left(-f(\beta, \sigma) \right)$. Since the function $\beta \mapsto s_{n, r}(\beta) - \beta^T C^{-1} \beta / 2$ is a convex function on $\R^d$, then by Lemma~\ref{lemma:subgradient_inequality} for every $\beta \in \R^d$,
\[
s_{n, r}(\beta) \ge s_{n, r}(\beta^*) + \frac{1}{2} \left( \beta - \beta^* \right)^T C^{-1} \left( \beta - \beta^* \right).
\]
For any $\left( \beta, \sigma \right) \in \R^d \times (0, \infty)$, we then have the lower bound
\begin{align*}
&f(\beta, \sigma)
\\
&=  \frac{s_0 + s_{n, r}(\beta)}{\sigma} + (n + v_0 + 1 + d/2) \log(\sigma)
\\
&\ge \frac{s_0 + s_{n, r}(\beta^*)}{\sigma} + (n + v_0 + 1 + d/2) \log(\sigma) + \frac{1}{2 \sigma} \left( \beta - \beta^* \right)^T C^{-1} \left( \beta - \beta^* \right).
\end{align*}
This implies
\begin{align*}
&f(\beta, \sigma) 
\\
&- \frac{1}{2 \sigma} \left( \beta - \beta^* \right)^T C^{-1} \left( \beta - \beta^* \right) - \frac{s_0 + s_{n, r}(\beta^*)}{\sigma} 
- (n + v_0 + 1 + d/2) \log(\sigma)
\\
&\ge 0.
\end{align*}
Let $q$ denote the proposal's normal-inverse-gamma density.
For any $\sigma_0 \in (0, \infty)$ and for every $\left( \beta, \sigma \right) \in \R^d \times (0, \infty)$, we have shown
\begin{align*}
\frac{q(\beta, \sigma)}{\pi(\beta, \sigma)}
&\ge Z_{\Pi(\cdot | X, Y)} \left( 2 \pi \right)^{-\frac{d}{2}} \det\left( C \right)^{-1/2} \left( s_0 + s_{n, r}(\beta^*) \right)^{n + v_0} \Gamma(n + v_0)^{-1}
\\
&= \frac{q(\beta^*, \sigma_0)}{\pi(\beta^*, \sigma_0)}.
\end{align*}
An application of Proposition \ref{prop:exactwasserstein} completes the proof.
\end{proof}

Note that $\e_{\beta^*}$ is difficult to compute since it depends on the normalizing constant, but we give an example later where upper bounding the convergence rate is possible in Bayesian logistic regression.

\section{The convergence rate at arbitrary initializations}
\label{section:convergence_at_every_point}

The previous section studies the worst-case convergence rate and the convergence rate at an individual point for the MHI sampler.
We can study the convergence rate at every point as was done for total variation \citep{Wang2022}.
The technique needed to prove this relies on the exact representation of the MHI sampler \citep[Theorem 1]{Smith1996} but new techniques are needed to show this in the Wasserstein distance.
Similar to the convergence rate in total variation \citep{Wang2022}, we define the Wasserstein convergence rate for a point $\theta \in \Theta$ as
\[
r_{\rho}(\theta)
= \lim_{t \to \infty} \W_{\rho}(P^t(\theta, \cdot), \Pi)^{1/t}.
\]

When the distance metric is $\min\{ \norm{\cdot}, 1 \}$ where $\norm{\cdot}$ can be any norm, Theorem~\ref{thm:rate_at_every_point} shows we can obtain the convergence rate at every point under mild conditions.
We require $\pi, q$ to be locally Lipschitz continuous and bounded which is stronger than only locally Lipschitz as in total variation \citep{Wang2022}.
However, this additional condition is satisfied in many practical applications in statistics.
Theorem~\ref{thm:rate_at_every_point} also lower bounds the convergence rate for Wasserstein $L_p$-distances and the rate of convergence for these distances cannot be improved.

\begin{theorem}
\label{thm:rate_at_every_point}
If $\pi, q$ are locally Lipschitz continuous and bounded on $\R^d$ and there exists a solution 
\[
\theta^* = \text{argmin}\left\{ q(\theta) / \pi(\theta) : \theta \in \Theta \right\},
\] 
then for any $\theta \in \Theta$, the Wasserstein convergence rate is the same with
\[
r_{ \min\{ \norm{\cdot}, 1 \} }(\theta)
= 1 - q(\theta^*) / \pi(\theta^*).
\]
\end{theorem}

\begin{proof}
If the initialization is at $\theta^*$, then the result follows from Proposition~\ref{prop:exactwasserstein}.
Fix a point $\theta_0 \in \Theta$ such that $\theta_0 \not= \theta^*$.
Using Lemma~\ref{lem:MHIS_UB}, we have an upper bound on the convergence rate by
\[
\limsup_{t \to \infty} \W_{\min\{ \norm{\cdot}, 1 \}}(P^t(\theta_0, \cdot), \Pi)^{1/t}
\le 1 - q(\theta^*)/\pi(\theta^*).
\]
It remains to lower bound the limit inferior.

For $h \in (0, 1]$, define the function 
\[
\phi_h(\theta) = \frac{1}{(2h)^d} \exp(-h^{-1} \norm{\theta - \theta^*}_1).
\]
This is the probability density function for a Laplace distribution and $\phi_h$ is $2^{-d} h^{-d - 1}$ Lipschitz, and so $2^{d} h^{d + 1} \phi_h(\theta) = h \exp(-h^{-1} \norm{\theta - \theta^*}_1)$ is nonnegative, $1-$Lipschitz and bounded by $1$.
In particular, it is readily shown that $2^{d} h^{d + 1} \phi_h(\theta)$ is $\min\{ 1, \norm{\cdot}_1 \}$ Lipschitz.
Using Kantorovich-Rubinstein duality \citep[Theorem 1.14]{Villani2003}, we have the lower bound
\begin{align}
\W_{\min\{ \norm{\cdot}_1, 1 \}}(P^t(\theta_0, \cdot), \Pi)
&\ge 
2^{d} h^{d + 1} \left[ \int \phi_h d\Pi - \int \phi_h dP^t(\theta_0, \cdot) \right] \label{eq:kr_exact_rate_lb}.
\end{align}

% Proof for locally Lipschitz
We will develop some approximation properties of $\phi_h$.
Since $\theta_0 \not= \theta^*$, then it is readily shown that $\lim_{h \downarrow 0} \phi_h(\theta_0) = 0$.
Using change of variables and since we have assumed $\sup_{\theta} \pi(\theta) < \infty$, then
\begin{align}
\left| \int \phi_h(\theta') \pi(\theta') d\theta' - \pi(\theta^*) \right|
&\le \int_{\norm{\theta'}_2 \le t} \left| \pi(\theta^* + h \theta') - \pi(\theta^*) \right| 
2^{-d} \exp(-\norm{\theta'}_1) d\theta'  \label{eq:mollifier_approx1}
\\
&\hspace{.4cm}+ 2 \sup_{\theta} \pi(\theta) \int_{\norm{\theta'}_2 > t} 2^{-d} \exp(-\norm{\theta'}_1) d\theta'. \label{eq:mollifier_approx2}
\end{align}
If $Y_1, \ldots, Y_d$ i.i.d. Laplace, then we have the tail bound
\begin{align*}
\P(\norm{Y}_2 \ge t)
\le \P\left( \max_{i} |Y_i| \ge \frac{t}{\sqrt{d}} \right)
&\le \sum_{i = 1}^d \P\left( |Y_i| \ge \frac{t}{\sqrt{d}} \right)
\\
&\le d \exp\left( -\frac{t}{\sqrt{d}} \right).
\end{align*}
Since norms are equivalent on $\R^d$, then $\pi, q$ are locally Lipschitz with respect to any norm.
Choosing $h = h_0/t^2$ for some $h_0 \in (0, 1)$, since $\pi$ is locally Lipschitz, we can find a universal constant $L \in (0, \infty)$ so that for any $\norm{\theta'}_2 \le t$.
\[
\left| \pi(\theta^* + h \theta') - \pi(\theta^*) \right|
\le \frac{h_0 L}{t}.
\]
Applying these upper bounds to \eqref{eq:mollifier_approx1} and \eqref{eq:mollifier_approx2}, then for large enough $t$,
\begin{align}
\left| \int \phi_h(\theta') \pi(\theta') d\theta' - \pi(\theta^*) \right|
&\le \frac{2 h_0 L}{t}. \label{eq:mollifier_bound_pi}
\end{align}
A similar argument with the assumptions on $q$ yields for large enough $t$,
\begin{align}
\left| \int \phi_h(\theta') q(\theta') d\theta' - q(\theta^*) \right|
&\le \frac{2 h_0 L}{t}.
\label{eq:mollifier_bound_q}
\end{align}

It remains to lower bound \eqref{eq:kr_exact_rate_lb}.
We will use the exact representation of the independence sampler \citep[Theorem 1, Lemma 3]{Smith1996}.
Define the importance sampling weight by $w(\theta) = \pi(\theta) / q(\theta)$ and its maximum $w^* = w(\theta^*) = \pi(\theta^*) / q(\theta^*)$.
The rejection probability can be represented for $\theta \in \Theta$ by
\[
\lambda(w(\theta))
= \int_{w(\theta') \le w(\theta)} \left[ 1 - \frac{w(\theta')}{w(\theta)}\right] q(\theta') d\theta'.
\]
For $w \in (0, \infty)$, define 
\[
T_t(w)
= \int_{w}^{\infty} \frac{t \lambda^{t-1}(v)}{v^2} dv,
\]
and using the exact representation of the independence sampler \citep[Theorem 1, Lemma 3]{Smith1996}, for measurable sets $B \subset \R^d$
\[
P^t(\theta_0, B)
= \int_B T_t( \max\{ w(\theta_0), w(\theta') \}) \pi(\theta') d\theta'
+ \lambda^t(w(\theta_0)) \delta_{\theta_0}(B).
\]
The proof of existing results \citep[Theorem 4]{Wang2022} shows that
\[
T_t(w)
\le 1 + (1 - 1/w^*)^t \left[ \frac{t}{w^* - 1} \left( \frac{w^*}{w} - 1 \right) - 1 \right].
\]

We now use the exact representation of the independence sampler to lower bound \eqref{eq:kr_exact_rate_lb}.
We have $\lambda^t(w(\theta_0)) \le (1 - 1/w^*)^t$ and so we then have the upper bound
\begin{align}
&\int \phi_h dP^t(\theta_0, \cdot) - \int \phi_h d\Pi \nonumber
\\
&\le -(1 - 1/w^*)^t \int \phi_h d\Pi
+ \lambda^t(w(\theta_0)) \phi_h(\theta_0) \nonumber
\\
&\hspace{.4cm}+ (1 - 1/w^*)^t \frac{t}{w^* - 1} \int \left[ \frac{w^*}{ \max\{ w(\theta_0) , w(\theta') \}} - 1 \right] \phi_h(\theta') \pi(\theta') d\theta' \nonumber
\\
&\le (1 - 1/w^*)^t \left[ \phi_h(\theta_0) - \int \phi_h d\Pi \right] \nonumber
\\
&\hspace{.4cm}+ (1 - 1/w^*)^t \frac{t}{w^* - 1} \left[ w^* \int \phi_h(\theta') q(\theta') d\theta' 
- \int \phi_h(\theta') \pi(\theta') d\theta' \right]. \label{eq:ub_exact_first}
\end{align}
Using the upper bound \eqref{eq:ub_exact_first} with the approximations \eqref{eq:mollifier_bound_pi} and \eqref{eq:mollifier_bound_q} yields
\begin{align*}
&\int \phi_h dP^t(\theta_0, \cdot) - \int \phi_h d\Pi
\\
&\le (1 - 1/w^*)^t \left[ \phi_h(\theta_0) - \pi(\theta^*) + \frac{2 h_0 L}{t} \right]
\\
&\hspace{.4cm}+ (1 - 1/w^*)^t \frac{t}{w^* - 1} \left[ w^* q(\theta^*) - \pi(\theta^*) 
+ (w^* + 1) \frac{2 h_0 L}{t} \right]
\\
&\le (1 - 1/w^*)^t \left[ \phi_h(\theta_0) - \pi(\theta^*) + \frac{2 h_0 L}{t} + \frac{w^* + 1}{w^* - 1} 2 h_0 L \right].
\end{align*}
Therefore, we can choose a small enough $h_0 \in (0, 1)$ independently of $t$ so that we have the upper bound
\begin{align*}
\int \phi_h dP^t(\theta_0, \cdot) - \int \phi_h d\Pi
&\le -\frac{\pi(\theta^*)}{4} (1 - 1/w^*)^t.
\end{align*}
Applying these bounds to \eqref{eq:kr_exact_rate_lb} and using that we have chosen $h = h_0/t^2$, we have the lower bound
\begin{align*}
\W_{ \min\{ \norm{\cdot}_1, 1 \}}(P^t(\theta_0, \cdot), \Pi)
&\ge 
2^{d} h^{d + 1} \left[ \int \phi_h d\Pi - \int \phi_h dP^t(\theta_0, \cdot) \right]
\\
&\ge 2^{d - 2} h^{d + 1} \pi(\theta^*) (1 - 1/w^*)^t
\\
&\ge 2^{d - 2} \left( \frac{h_0}{t^2} \right)^{d + 1} \pi(\theta^*) (1 - 1/w^*)^t.
\end{align*}
%Since $\lim_{t \to \infty} t^{-2/t} = 1$, then 
Taking the limit
\[
\liminf_{t \to \infty} \W_{ \min\{ \norm{\cdot}_1, 1 \}}(P^t(\theta_0, \cdot), \Pi)^{1/t}
\ge 1 - 1/w^*
= 1 - q(\theta^*) / \pi(\theta^*).
\]
Since all norms are equivalent on $\R^d$, this can be extended to any norm $\norm{\cdot}$.
\end{proof}

\section{MHI samplers with centered Gaussian proposals}
\label{section:4}

We look to apply the exact convergence expression from Proposition~\ref{prop:exactwasserstein} to practical applications since the convergence rate is the same at every initialization under mild assumptions. 
Recently, centered drift functions have been used to improve convergence analyses of some Monte Carlo Markov chains \cite{KarlOskar2019, qin2019convergence, qin2021wasserstein}. Our focus is instead on centering the proposal distribution, that is, matching the optimal points of the proposal and target densities similar to Laplace approximations.

We shall see in the next section that by centering a Gaussian proposal, we may satisfy the assumptions of Proposition~\ref{prop:exactwasserstein} for a general class of target distributions with $\theta^*$ being the optimum of the target's density.  While we focus on Gaussian proposals,
the technique of centering proposals is in fact more general.

We will assume the target
distribution $\Pi$ is a probability distribution supported on $\R^d$.
With $f : \R^d \to \R$ and normalizing constant $Z_{\Pi}$, define the density $\pi$ by
$\pi(\theta) = Z_{\Pi}^{-1} \exp\left( -f(\theta) \right)$. 
Let $\theta^*$ be the unique maximum of $\pi$, $\alpha \in (0, +\infty)$, and $C \in \R^{d \times d}$ be a symmetric, positive-definite matrix. 
Let the proposal distribution $Q$ with density $q$ correspond to a $d$-dimensional Gaussian distribution, $\text{N}_{d}(\theta^*, \alpha^{-1} C )$. In this case, the ratio of the proposal density and target density is $\e_{\theta^*} = (2 \pi)^{-d/2} \alpha^{d/2} \det\left( C \right)^{-1/2} Z_{\Pi} \exp(f(\theta^*))$.

\begin{proposition}
 \label{theorem:gaussianexactwasserstein}
  If $\theta^*$ exists and, for
  any $\theta \in \R^d$,
\[
f(\theta) \ge f(\theta^*) + \alpha \left( \theta -
    \theta^* \right)^T C^{-1} \left( \theta - \theta^* \right)/2,
\]
then
\[
\W_{\rho}\left( P^t( \theta^*, \cdot ), \Pi \right)
= \left( 1 - \e_{\theta^*} \right)^t \int \rho(\theta, \theta^*) d\Pi(\theta).
\]
\end{proposition}

\begin{proof}
  Since the proposal density has been centered at the point
  $\theta^*$, it then satisfies
  $q(\theta^*) = (2 \pi)^{-d/2} \alpha^{d/2} \det\left( C \right)^{-1/2}$.  For every
  $\theta \in \R^d$, we have the lower bound
\begin{align*}
  \frac{q(\theta)}{\pi(\theta)}
  &= (2 \pi)^{-d/2} \alpha^{d/2} \det\left( C \right)^{-1/2} Z_{\Pi} \exp\left( f(\theta) - \frac{\alpha}{2}
    \left( \theta - \theta^*
    \right)^T C^{-1} \left( \theta -
    \theta^* \right) \right) 
  \\
  &\ge (2 \pi)^{-d/2} \alpha^{d/2} \det\left( C \right)^{-1/2} Z_{\Pi} \exp\left( f(\theta^*) \right)
  \\
  &= \frac{q(\theta^*)}{\pi(\theta^*)}.
\end{align*}

Since both densities are positive and the proposal is independent of
the previous iteration, we have shown that the conditions for Proposition~\ref{prop:exactwasserstein} are satisfied and an application of Proposition
\ref{prop:exactwasserstein} with the proposal and target distribution $Q$ and $\Pi$ as we have defined them in this section completes the proof.
\end{proof}
 
The point $\theta^*$ is guaranteed to exist if the function $f$ satisfies
a convexity property.
A function $h : \R^d \to \R$ is strongly convex with parameter $\mu$ if there is an $\mu \in (0, +\infty)$ so that
$h(\cdot) - \mu \norm{\cdot}^2 / 2$ is convex
\citep{Urruty2001, Nesterov2018}.  The norm in this definition is
often taken to be the Euclidean norm, but we will use the
norm induced by the matrix $C^{-1}$. 
We consider using a Gaussian proposal centered at a point $\theta_0$ which is not necessarily the optimum of the target density.
Let $g_{f(\theta_0)} \in \R^d$ be a subgradient of $f$ at $\theta_0$ \citep{Nesterov2018}.
For a point $\theta_0 \in \R^d$, we consider the proposal corresponding to a $d$-dimensional Gaussian distribution, $\text{N}_{d}(\theta_0 - \alpha^{-1} C g_{f(\theta_0)}, \alpha^{-1} C)$.  When $f$ is differentiable, this construction of the proposal uses the gradient of $f$ in a similar way as MALA. 
The ratio of the proposal and target density evaluated at $\theta_0$ is $\e_{\theta_0} = (2 \pi)^{-d/2} \det\left( \alpha^{-1} C \right)^{-1/2} Z_{\Pi} \exp(f(\theta_0) - {g_{f(\theta_0)}}^T C g_{f(\theta_0)} / (2\alpha) )$.
Choosing $\theta_0 \equiv \theta^*$ yields the centered Gaussian proposal, but we also have an exact convergence expression in other cases.

\begin{proposition}
\label{theorem:stronglyconvexexactwasserstein}
If the function
  $\theta \mapsto f(\theta) - \alpha \theta^T C^{-1} \theta / 2$ is convex for
  all points on $\R^d$, then
\[
\W_{\rho}\left( P^t(\theta_0, \cdot), \Pi \right)
= (1 - \e_{\theta_0})^t \int \rho(\theta, \theta_0) d\Pi(\theta).
\]
\end{proposition}

\begin{proof}
Since the function
  $f(\theta) - \alpha \theta^T C^{-1} \theta/2$ is convex for
  all points on $\R^d$, then for each $\theta \in \R^d$,
\begin{align*}
f(\theta)
&\ge f(\theta_0) 
+ {g_{f(\theta_0)}}^T (\theta - \theta_0)
+ \frac{\alpha}{2} \left( \theta - \theta_0 \right)^T C^{-1} (\theta - \theta_0)
\\
&= f(\theta_0) - \frac{1}{2 \alpha} (C g_{f(\theta_0)})^T g_{f(\theta_0)}
\\
&\hspace{.5cm}+ \frac{\alpha}{2} \left( \theta - \theta_0 +  \alpha^{-1} C g_{f(\theta_0)} \right)^T C^{-1} \left( \theta - \theta_0 + \alpha^{-1} C g_{f(\theta_0)} \right).
\end{align*}
This implies for every $\theta \in \R^d$, the ratio of the proposal density $q$ corresponding to the distribution $\text{N}_{d}(\theta_0 - \alpha^{-1} C g_{f(\theta_0)}, \alpha^{-1} C)$ and target density $\pi$ satisfies
\[
\frac{q(\theta)}{\pi(\theta)}
\ge \frac{q(\theta_0)}{\pi(\theta_0)}
= \e_{\theta_0}.
\]
An application of Proposition
\ref{prop:exactwasserstein} completes the proof.
\end{proof}

\subsection{Application: Bayesian generalized linear models}

We consider Bayesian Poisson and negative-binomial regression for count response regression and Bayesian logistic and probit regression for binary response regression (e.g. logistic and probit).
Suppose there are $n$ discrete-valued responses $Y_i$
  with features $X_i \in \R^d$, and a parameter $\beta \in \R^d$. 
For Poisson regression, assume the $Y_i$'s are conditionally independent with
\[
Y_i \big| X_i, \beta \sim \text{Poisson}\left( \exp\left( \beta^T X_i \right) \right).
\] 
Similarly, for negative-binomial regression, if $\xi \in (0, +\infty)$, assume
\[
Y_i \big| X_i, \beta \sim \text{Negative-Binomial}\left( \xi, \left( 1 + \exp\left( -\beta^T X_i \right) \right)^{-1} \right).
\]
For binary response regression, if $S : \R \to (0, 1)$, assume 
\[
Y_{i} | X_{i}, \beta \sim \text{Bernoulli} \left( S\left( \beta^T X_i \right)\right).
\]
For logistic regression, we will consider $S(x) = \left( 1 + \exp\left( x \right) \right)^{-1}$ and for probit regression, we will consider $S(x)$ to be the cumulative distribution function of a standard Gaussian random variable.
Suppose $\beta \sim \text{N}_{d}(0, \alpha^{-1} C)$ where
$\alpha \in (0, +\infty)$ and $C \in \R^{d \times d}$ is a symmetric,
positive-definite matrix. Both $\alpha$ and $C$ are assumed to be known.
Define the vector $Y = \left( Y_1, \ldots, Y_n \right)^T$ and the matrix $X = \left( X_1, \ldots, X_n \right)^T$. 
Let $\Pi( \cdot | X, Y)$ denote the posterior with density $\pi( \cdot | X, Y)$. 
If $\l_n$ denotes the negative log-likelihood for each model, the posterior density
is characterized by
\[
\pi( \beta | X, Y)
= Z_{\Pi(\cdot  | X, Y))}^{-1} \exp\left(-\l_n(\beta) - \frac{\alpha}{2}
    \beta^T C^{-1} \beta \right).
\]
Observe that the function $\l_n$ is convex in all four models we consider.
Let $\beta^*$ denote the unique maximum of $\pi(\cdot | X, Y)$.
For the MHI
algorithm, we use a $\text{N}_{d}(\beta^*, \alpha^{-1} C)$ proposal
distribution, and Proposition~\ref{theorem:stronglyconvexexactwasserstein} immediately yields the following for each posterior. 

\begin{corollary}\label{corollary:logisticexactwasserstein}
We have
\[
\W_{\rho}\left( P^t(\beta^*, \cdot), \Pi(\cdot | X, Y) \right)
= \left( 1 - \e_{\beta^*} \right)^t \int \rho(\beta, \beta^*) d\Pi\left( \beta | X, Y \right)
\]
where $\e_{\beta^*} = \exp(\l_n(\beta^*) + \frac{\alpha}{2}
    {\beta^*}^T C^{-1} \beta^*) Z_{\Pi(\cdot | X, Y)} \left( (2 \pi)^{d/2} \det\left( \alpha^{-1} C \right)^{1/2} \right)^{-1}$.
\end{corollary}

\subsection{Convergence complexity analysis in binary response regression}

Our goal now is to obtain an upper bound on the rate of convergence established in Corollary \ref{corollary:logisticexactwasserstein}
in high dimensions for binary response regression.  In this context, it is natural to treat the
$(Y_i, X_i)_{i = 1}^n$ as stochastic; each time the sample size increases, the additional observation is randomly generated.  Specifically, we  will assume that $(Y_i, X_i)_{i = 1}^n$ are independent with
  $Y_{i} | X_{i}, \beta \sim \text{Bernoulli} \left( S\left( \beta^T X_i \right)\right)$ and $X_i \sim N_{d}(0, \sigma^2 n^{-1} I_d)$ with
  $\sigma^2 \in (0, +\infty)$.
Similar scaling assumptions on the data are used for high-dimensional
maximum-likelihood theory in logistic regression \citep{Sur2019}.  We
will also assume the limit of the trace of the covariance matrix used
in our prior is finite, that is, $tr(C) \to s_0 \in (0, +\infty)$ as $d \to +\infty$. Note that this is a necessary condition that the trace of the covariance is finite for Gaussian distributions to exist in an infinite-dimensional Hilbert space \citep{Bogochev2012}.

\begin{theorem}
\label{theorem:asympbound} 
Suppose that the following conditions hold. 

\begin{enumerate}
\item The negative log-likelihood $\l_n$ is a twice continuously differentiable convex function.

\item There is a universal constant $r_0 \in (0, +\infty)$ such that the largest eigenvalue of the Hessian of the negative log-likelihood  $H_{\l_n}$ satisfies for every $\beta \in \R^d$,
\[
\lambda_{max} \left( H_{\l_n}(\beta) \right) \le r_0 \lambda_{max}\left( X^T X \right).
\]
\end{enumerate}
\noindent Let
  $a_0 =  r_0 (1 + \gamma^{1/2})^2 \sigma^2 s_0 / (2 \alpha)$. If
  $d, n \to +\infty$ in such a way that $d/n \to \gamma \in (0, +\infty)$,
  then, almost surely
\[
\limsup_{d/n \to \gamma} 
\W_{\rho}\left( P^t(\beta^*, \cdot), \Pi( \cdot | X, Y) \right)
\le M_0 (1 - \exp\left( -a_0 \right))^t 
\]
where $M_0 = \limsup_{d/n \to \gamma} \int \rho(\beta, \beta^*) d\Pi( \beta |
X, Y)$.
\end{theorem}

\begin{proof}
  Under our assumption, we may write the matrix
  $X = n^{-1/2} G$ where $G$ is a matrix with i.i.d Gaussian
  entries with mean $0$ and variance $\sigma^2$. Denote the
  largest eigenvalue of the matrix $X^T X$ by $\lambda_{max}(X^T X)$.
  Therefore, as $d, n \to \infty$ in such a way that
  $d/n \to \gamma \in (0, +\infty)$,
\[
\lambda_{max}\left( X^T X \right) 
= \lambda_{max}\left( \frac{1}{n} G^T G \right) 
= \frac{1}{n} \sup_{v \in \R^d, \norm{v}_2 = 1} \norm{G^T G v}_2
\to (1 + \gamma^{1/2})^2 \sigma^2 
\]
almost surely \cite[][Theorem 1]{Geman1980}.

Define the function $f : \R^d \to \R$ by
$f(\beta) = \l_n(\beta) + \alpha / 2 \beta^T C^{-1} \beta$ where
$\l_n$ is the negative log-likelihood loss function and define $Z_Q = (2 \pi)^{d/2} \det\left( \alpha^{-1} C \right)^{1/2}$. We will first
lower bound the quantity
$\exp(f(\beta^*)) Z_{\Pi( \cdot | X, Y)} / Z_{Q}$. 
We have that for any $\beta \in \R^d$ and any $v \in \R^d$,
\[
v^T H_{\l_n}(\beta) v 
\le r_0 \lambda_{max}(X^T X) \norm{v}_2^2.
\]
This implies that for any $\beta \in \R^d$ and any $v \in \R^d$, the Hessian of $f$, denoted by $H_{f}$, satisfies
\[
v^T H_{f}(\beta) v 
\le v^T \left( r_0 \lambda_{max}(X^T X) I_d + \alpha C^{-1} \right) v.
\]
Since the function $\l_n$ is twice continuously differentiable, then $f$ is also twice continuously differentiable.
Since both the gradient $\nabla f$ and Hessian $H_{f}$ are continuous and $\nabla f(\beta^*) = 0$, we use a Taylor expansion to obtain the upper bound
\begin{align*}
f(\beta)  
&= f(\beta^*) + \int_0^1 \int_0^t (\beta - \beta^*)^T H_{f}(\beta^* + s (\beta - \beta^*)) (\beta - \beta^*) ds dt 
\\
&\le f(\beta^*) + \frac{1}{2} (\beta - \beta^*)^T \left(
     r_0 \lambda_{max}(X^T X) I_d + \alpha C^{-1} \right) (\beta - \beta^*). 
\end{align*}
We then have a lower bound on the normalizing constant of the target posterior
\[
Z_{\Pi(\cdot | X, Y)}
= \int_{\R^d} \exp\left( -f(\beta) \right) d\beta
\ge \frac{\exp\left( -f(\beta^*) \right) (2 \pi)^{d/2}}{\det\left( r_0 \lambda_{max}(X^T X) I_d + \alpha C^{-1} \right)^{1/2}}.
\]
This in turn yields a lower bound on the ratio
\begin{align}
\frac{Z_{\Pi(\cdot | X, Y)}}{Z_Q} \exp\left( f(\beta^*) \right)
\ge \frac{\det\left( \alpha C^{-1} \right)^{1/2}}{\det\left(
  r_0 \lambda_{max}(X^T X) I_d  + \alpha C^{-1} \right)^{1/2}}. \label{eq:ratio}
\end{align}
The matrix $C$ is symmetric and positive-definite and so its
eigenvalues exist and are positive.  Let
$\left( \lambda_i(C) \right)_{i = 1}^d$ be the eigenvalues of $C$.  It
is readily verified that the eigenvalues of the matrix
$r_0 \lambda_{max}(X^T X) I_d + \alpha C^{-1}$ exist and are
$\left( r_0 \lambda_{max}(X^T X) + \frac{\alpha}{\lambda_i(C)}
\right)_{i = 1}^d$. Then
\begin{align}
&\frac{\det\left( \alpha C^{-1} \right)}{\det\left( r_0 \lambda_{max}(X^T X) I_d  + \alpha C^{-1} \right)}
\\
&= \frac{\prod_{i = 1}^d \frac{\alpha}{\lambda_i(C)}}{ \prod_{i = 1}^d \left( r_0 \lambda_{max}\left( X^T X \right) + \frac{\alpha}{\lambda_i(C)} \right)} \nonumber
\\
&= \prod_{i = 1}^d \frac{\frac{\alpha}{\lambda_i(C)}}{ r_0 \lambda_{max}\left( X^T X \right) + \frac{\alpha}{\lambda_i(C)}} \nonumber
\\
&= \prod_{i = 1}^d \frac{1}{\frac{r_0}{\alpha} \lambda_{max}\left( X^T X \right) \lambda_i(C) + 1} \nonumber
\\
&= \exp\left( -\sum_{i = 1}^d \log\left( \frac{r_0}{\alpha} \lambda_{max}\left( X^T X \right)\lambda_i(C) + 1 \right) \right). \label{eq:1}
\end{align}

We have the basic inequality $\log(x + 1) \le x$ for any $x \in [0, +\infty)$.
Since the eigenvalues of $C$ are positive and $\lambda_{max}\left( X^T X \right)$ is nonnegative, we have the upper bound
\begin{align}
\sum_{i = 1}^d \log\left( \frac{r_0}{\alpha} \lambda_{max}\left( X^T X \right)\lambda_i(C) + 1 \right)
\le \frac{r_0}{\alpha} \lambda_{max}\left( X^T X \right) \sum_{i = 1}^d \lambda_i(C) \label{eq:3}.
\end{align}
This yields a lower bound on \eqref{eq:1}.  Define the doubly-indexed
sequence $(a_{d, n})$ by
  \[
    a_{d, n} =  \frac{r_0}{2\alpha} \lambda_{max}\left( X^T X
  \right) \sum_{i = 1}^d \lambda_i(C).
\]
We have then shown that
\begin{align}
\frac{Z_{\Pi(\cdot | X, Y)}}{Z_Q} \exp\left( f(\beta^*) \right)
\ge \exp\left( -a_{d, n} \right).
\end{align}

By our assumption, $tr(C) \to s_0$ as $d \to \infty$.
That is to say that 
\[
\lim_{d \to +\infty} \sum_{i = 1}^d \lambda_i(C) = s_0.
\]
Then as $d, n \to \infty$ in such a way that $d / n \to \gamma \in (0, +\infty)$, by continuity
\[
a_{d, n}
\to \frac{r_0}{2 \alpha} (1 + \gamma^{1/2})^2 \sigma^2 s_0
\]
almost surely.
This implies using continuity that almost surely,
\begin{align*}
\lim_{\substack{d, n \to \infty 
\\ \frac{d}{n} \to \gamma}} (1 - \exp\left( -a_{n, d} \right))^t
&= (1 - \exp\left( -a_0 \right))^t.
\end{align*}

By Corollary
\ref{corollary:logisticexactwasserstein}, we have the upper bound on
the Wasserstein distance for each $d$ and each $n$:
\begin{align*}
&\W_{\rho}\left( P^t(\beta^*, \cdot), \Pi( \cdot | X, Y) \right)
\\
&= \left( 1 -  \exp(f(\beta^*)) \frac{Z_{\Pi( \cdot | X, Y)}}{Z_Q} \right)^t \int_{\R^d} \rho(\beta, \beta^*) d\Pi( \beta | X, Y)
\\
&\le (1 - \exp\left( -a_{n, d} \right))^t \int_{\R^d} \rho(\beta, \beta^*) d\Pi( \beta | X, Y).
\end{align*}
Suppose that $\limsup_{\substack{d, n \to \infty 
\\ \frac{d}{n} \to \gamma}} \int_{\R^d} \rho(\beta, \beta^*) d\Pi( \beta | X, Y) < \infty$.
Using properties of the limit superior,
\begin{align*} 
&\limsup_{\substack{d, n \to \infty 
\\ \frac{d}{n} \to \gamma}} \W_{\rho}\left( P^t(\beta^*, \cdot), \Pi( \cdot | X, Y) \right)
\\
&\le \limsup_{\substack{d, n \to \infty 
\\ \frac{d}{n} \to \gamma}} (1 - \exp\left( -a_{n, d} \right))^t \limsup_{\substack{d, n \to \infty 
\\ \frac{d}{n} \to \gamma}} \int_{\R^d} \rho(\beta, \beta^*) d\Pi( \beta | X, Y)
\\
&= \lim_{\substack{d, n \to \infty 
\\ \frac{d}{n} \to \gamma}} (1 - \exp\left( -a_{n, d} \right))^t \limsup_{\substack{d, n \to \infty 
\\ \frac{d}{n} \to \gamma}} \int_{\R^d} \rho(\beta, \beta^*) d\Pi( \beta | X, Y)
\\
&= (1 - \exp\left( -a_0 \right))^t \limsup_{\substack{d, n \to \infty 
\\ \frac{d}{n} \to \gamma}} \int_{\R^d} \rho(\beta, \beta^*) d\Pi( \beta | X, Y).
\end{align*}

The other case when $\limsup_{\substack{d, n \to \infty 
\\ \frac{d}{n} \to \gamma}} \int_{\R^d} \rho(\beta, \beta^*) d\Pi( \beta | X, Y) = +\infty$ is trivial.
\end{proof}

Theorem~\ref{theorem:asympbound} applies to both Bayesian logistic and probit regression.
For logistic regression, $\l_n$ is a twice continuously differentiable convex function and we may choose $r_0 = 4^{-1}$.
Similarly for probit regression, $\l_n$ is also a twice continuously differentiable convex function and we may choose $r_0 = 1$ \citep{Demidenko2001}.

In Figure \ref{figure:logisticrate}, we plot $(1 - \exp\left( -a_0 \right))^t$,
the limiting decrease in
the Wasserstein distance according to our upper bound,
with varying values of the limiting ratio $\gamma$ with the other remaining values in $a_0$ fixed.  
We observe that as this ratio increases, the number of iterations needed to approximately converge may still increase rather rapidly.

\begin{figure}[ht!]
\centering
\includegraphics[width=.6\linewidth]{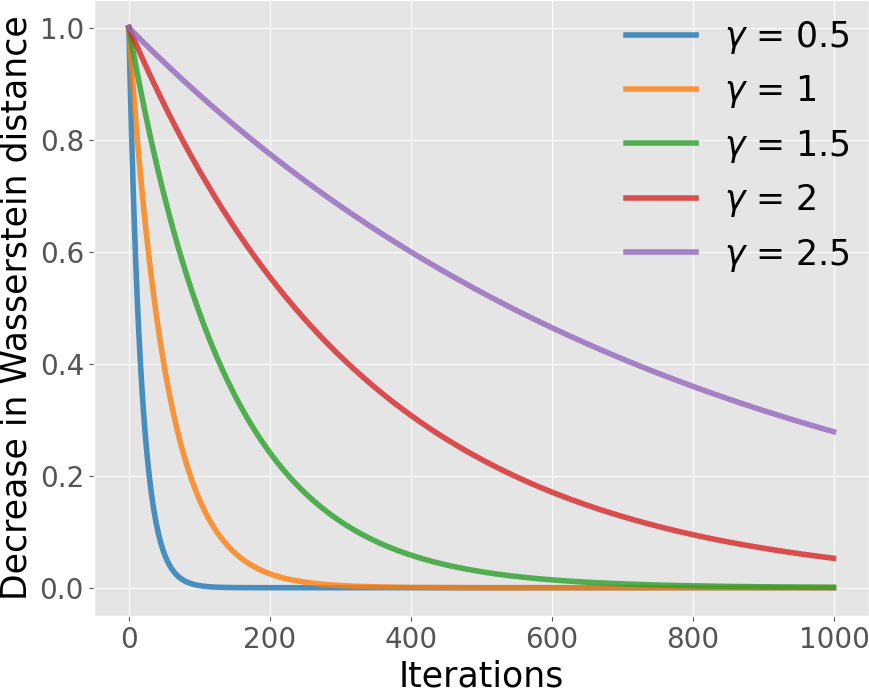}
\caption[Limiting decrease in the Wasserstein distance]{The limiting decrease in the Wasserstein distance using
  different values of $\gamma$, the limiting ratio of the dimension and
  sample size, versus the number of
  iterations.}
\label{figure:logisticrate} 
\end{figure}

\section{Final remarks}\label{section:conclusion}

We studied the exact convergence behavior of the MHI sampler across general Wasserstein distances.
We showed upper and lower bounds on the worst-case convergence rate for Wasserstein distances weaker than the total variation distance.
We showed the exact convergence rate at every initialization for Wasserstein distances weaker than the total variation distance is the same and matches that of the total variation convergence rate \citep{Wang2022}.
When starting at a certain point, we gave exact convergence expressions.
%It remains an open question if the convergence rate is the same for every initial starting point as it is for total variation \citep{Wang2022}.
By centering an independent proposal, we directly imbue the Markov chain with a strong attraction to a set where the target distribution has high probability. We showed this technique can provide uniform control over acceptance probability yielding exact convergence rates in Bayesian quantile regression.
The centered MHI sampler turns out to have many applications for posteriors that arise in Bayesian generalized linear models where more sophisticated proposals are often used. With additional assumptions on the data and prior, we also showed this exact convergence rate may be upper bounded when sampling high-dimensional posteriors in Bayesian binary response regression. 
%Explicit convergence rates for Metropolis-Hastings algorithms such as MALA and RWM are difficult to obtain especially in high dimensions.
%Although connections between the MHI sampler and rejection sampling are well-known \citep{Liu1996, Tierney1994, Wang2022}, the centered MHI sampler may provide insight into the convergence behavior of more popular Metropolis-Hastings algorithms such as the MALA and RWM algorithms.

\bibliographystyle{plain}
\bibliography{mhis.bib}

\begin{thebibliography}{10}

\bibitem{Chib1993}
James~H. Albert and Siddhartha Chib.
\newblock {B}ayesian analysis of binary and polychotomous response data.
\newblock {\em Journal of the American Statistical Association},
  88(422):669--679, 1993.

\bibitem{Belloni2009}
Alexandre Belloni and Victor Chernozhukov.
\newblock On the computational complexity of {MCMC}-based estimators in large
  samples.
\newblock {\em The Annals of Statistics}, 37(4):2011--2055, 2009.

\bibitem{Bogochev2012}
Vladimir~I. Bogachev.
\newblock {\em {G}aussian Measures}.
\newblock American Mathematical Society, 1998.

\bibitem{Brooks2011}
Steve Brooks, Andrew Gelman, Galin~L. Jones, and Xiao-Li Meng.
\newblock {\em Handbook of {M}arkov chain {M}onte {C}arlo}.
\newblock Chapman and Hall/CRC, 1 edition, 2011.

\bibitem{Demidenko2001}
Eugene Demidenko.
\newblock Computational aspects of probit model.
\newblock {\em Mathematical Communications}, 6:233--247, 2001.

\bibitem{Durmus2015}
Alain Durmus and {\'E}ric Moulines.
\newblock Quantitative bounds of convergence for geometrically ergodic {M}arkov
  chain in the {W}asserstein distance with application to the {M}etropolis
  adjusted {L}angevin algorithm.
\newblock {\em Statistics and Computing}, 25:5--19, 2015.

\bibitem{Durmus2019}
Alain Durmus and {\'E}ric Moulines.
\newblock High-dimensional {B}ayesian inference via the unadjusted {L}angevin
  algorithm.
\newblock {\em Bernoulli}, 25:2854--2882, 2019.

\bibitem{Dwivedi2018}
Raaz Dwivedi, Yuansi Chen, Martin~J Wainwright, and Bin Yu.
\newblock {Log-concave sampling: Metropolis-Hastings algorithms are fast!}
\newblock In {\em Proceedings of the 31st Conference On Learning Theory},
  volume~75 of {\em Proceedings of Machine Learning Research}, pages 793--797,
  2018.

\bibitem{Eberle2014}
Andreas Eberle.
\newblock {Error bounds for Metropolis–Hastings algorithms applied to
  perturbations of Gaussian measures in high dimensions}.
\newblock {\em The Annals of Applied Probability}, 24(1):337 -- 377, 2014.

\bibitem{KarlOskar2019}
Karl~Oskar Ekvall and Galin~L. Jones.
\newblock Convergence analysis of a collapsed {G}ibbs sampler for {B}ayesian
  vector autoregressions.
\newblock {\em Electronic Journal of Statistics}, 15:691--721, 2021.

\bibitem{Geman1980}
Stuart Geman.
\newblock A limit theorem for the norm of random matrices.
\newblock {\em The Annals of Probability}, 8:252--261, 1980.

\bibitem{Gibbs2004}
Alison~L. Gibbs.
\newblock Convergence in the {W}asserstein metric for {M}arkov chain {M}onte
  {C}arlo algorithms with applications to image restoration.
\newblock {\em Stochastic Models}, 20(4):473--492, 2004.

\bibitem{Giraudo}
Davide Giraudo.
\newblock Product measure with a dirac delta marginal.
\newblock Mathematics Stack Exchange, 2014.

\bibitem{Hairer2014}
Martin Hairer, Andrew~M. Stuart, and Sebastian~J. Vollmer.
\newblock Spectral gaps for a {M}etropolis–{H}astings algorithm in infinite
  dimensions.
\newblock {\em The Annals of Applied Probability}, 24:2455--2490, 2014.

\bibitem{Hastings1970}
Wilfred~K. Hastings.
\newblock Monte {C}arlo sampling methods using {M}arkov chains and their
  applications.
\newblock {\em Biometrika}, 57:97--109, 1970.

\bibitem{Urruty2001}
Jean-Baptiste Hiriart-Urruty and Claude Lema{\'e}chal.
\newblock {\em Fundamentals of Convex Analysis}.
\newblock Springer-Verlag Berlin Heidelberg, 1 edition, 2001.

\bibitem{Jarner2000}
Søren~Fiig Jarner and Ernst Hansen.
\newblock Geometric ergodicity of {M}etropolis algorithms.
\newblock {\em Stochastic Processes and their Applications}, 85:341--361, 2000.

\bibitem{jin:tan:2020}
Rui Jin and Aixin Tan.
\newblock Central limit theorems for {M}arkov chains based on their convergence
  rates in {W}asserstein distance.
\newblock {\em preprint arXiv:2002.09427}, 2020.

\bibitem{john:etal:2019}
James~E. Johndrow, Aaron Smith, Natesh Pillai, and David~B. Dunson.
\newblock {MCMC} for imbalanced categorical data.
\newblock {\em Journal of the American Statistical Association},
  114:1394--1403, 2019.

\bibitem{Johnson2012}
Leif~T. Johnson and Charles~J. Geyer.
\newblock {Variable transformation to obtain geometric ergodicity in the
  random-walk Metropolis algorithm}.
\newblock {\em The Annals of Statistics}, 40(6):3050 -- 3076, 2012.

\bibitem{Jones2004}
Galin~L. Jones.
\newblock On the {M}arkov chain central limit theorem.
\newblock {\em Probability Surveys}, 1:299--320, 2004.

\bibitem{Joulin2010}
Ald{\'e}ric Joulin and Yann Ollivier.
\newblock Curvature, concentration and error estimates for {M}arkov chain
  {M}onte {C}arlo.
\newblock {\em The Annals of Probability}, 38:2418--2442, 2010.

\bibitem{Kantorovich1957}
L.~V. Kantorovich and G.~S. Rubinstein.
\newblock On a function space in certain extremal problems.
\newblock {\em Dokl. Akad. Nauk USSR}, 115(6):1058--1061, 1957.

\bibitem{Khare2012}
Kshitij Khare and James~P. Hobert.
\newblock Geometric ergodicity of the {G}ibbs sampler for {B}ayesian quantile
  regression.
\newblock {\em Journal of Multivariate Analysis}, 112:108--116, 2012.

\bibitem{Komorowski2011}
Tomasz Komorowski and Anna Walczuk.
\newblock Central limit theorem for {M}arkov processes with spectral gap in the
  {W}asserstein metric.
\newblock {\em Stochastic Processes and their Applications}, 122:2155--2184,
  2011.

\bibitem{Liu1996}
Jun~S. Liu.
\newblock Metropolized independent sampling with comparisons to rejection
  sampling and importance sampling.
\newblock {\em Statistics and Computing}, 6(2):113--119, 1996.

\bibitem{Madras2010}
Neal Madras and Deniz Sezer.
\newblock {Quantitative bounds for Markov chain convergence: Wasserstein and
  total variation distances}.
\newblock {\em Bernoulli}, 16(3):882 -- 908, 2010.

\bibitem{Mengersen1996}
Kerrie~L. Mengersen and Richard~L. Tweedie.
\newblock Rates of convergence of the {H}astings and {M}etropolis algorithms.
\newblock {\em The Annals of Statistics}, 24:101--121, 1996.

\bibitem{Metropolis1953}
Nicholas Metropolis, Arianna~W. Rosenbluth, Marshall~N. Rosenbluth, Augusta~H.
  Teller, and Edward Teller.
\newblock Equation of state calculations by fast computing machines.
\newblock {\em The Journal of Chemical Physics}, 21:1087--1092, 1953.

\bibitem{Meyn2009}
Sean~P. Meyn and Richard~L. Tweedie.
\newblock {\em Markov Chains and Stochastic Stability}.
\newblock Cambridge University Press, USA, 2 edition, 2009.

\bibitem{Nesterov2018}
Yurii Nesterov.
\newblock {\em Lectures on Convex Optimization}.
\newblock Springer International Publishing, 2 edition, 2018.

\bibitem{Papaspiliopoulos2019}
Omiros Papaspiliopoulos, Gareth~O. Roberts, and Giacomo Zanella.
\newblock Scalable inference for crossed random effects models.
\newblock {\em Biometrika}, 107(1):25--40, 2019.

\bibitem{Papaspiliopoulos2021}
Omiros Papaspiliopoulos, Timothée Stumpf-Fétizon, and Giacomo Zanella.
\newblock Scalable computation for {B}ayesian hierarchical models.
\newblock {\em preprint arXiv:2103.10875}, 2021.

\bibitem{Pierre2011}
Jacob Pierre, Christian~P. Robert, and Murray~H. Smith.
\newblock Using parallel computation to improve independent
  {M}etropolis–{H}astings based estimation.
\newblock {\em Journal of Computational and Graphical Statistics},
  20(3):616--635, 2011.

\bibitem{Polson2013}
Nicholas~G. Polson, James~G. Scott, and Jesse Windle.
\newblock Bayesian inference for logistic models using {P}{\' o}lya–{G}amma
  latent variables.
\newblock {\em Journal of the American Statistical Association},
  108:1339--1349, 2013.

\bibitem{qin2019convergence}
Qian Qin and James~P Hobert.
\newblock Convergence complexity analysis of {A}lbert and {C}hib's algorithm
  for {B}ayesian probit regression.
\newblock {\em Annals of Statistics}, 47:2320--2347, 2019.

\bibitem{qin2021bounds}
Qian Qin and James~P Hobert.
\newblock Geometric convergence bounds for {M}arkov chains in {W}asserstein
  distance based on generalized drift and contraction conditions.
\newblock {\em {\rm To appear in} Annales de l'Instute Henri Poincar{\' e}},
  2021.

\bibitem{qin2021limitations}
Qian Qin and James~P Hobert.
\newblock On the limitations of single-step drift and minorization in {M}arkov
  chain convergence analysis.
\newblock {\em {\rm To appear in} Annals of Applied Probability}, 2021.

\bibitem{qin2021wasserstein}
Qian Qin and James~P Hobert.
\newblock Wasserstein-based methods for convergence complexity analysis of
  {MCMC} with applications.
\newblock {\em {\rm To appear in} Annals of Applied Probability}, 2021.

\bibitem{raja:spar:2015}
Bala Rajaratnam and Doug Sparks.
\newblock {MCMC}-based inference in the era of big data: {A} fundamental
  analysis of the convergence complexity of high-dimensional chains.
\newblock {\em preprint arXiv:1508.00947}, 2015.

\bibitem{Roberts1996geo}
Gareth~O. Roberts and Richard~L. Tweedie.
\newblock Geometric convergence and central limit theorems for multidimensional
  {H}astings and {M}etropolis algorithms.
\newblock {\em Biometrika}, 83:95--110, 1996.

\bibitem{robe:etal:viz:2021}
Nathan Robertson, James~M. Flegal, Dootika Vats, and Galin~L. Jones.
\newblock Assessing and visualizing simultaneous simulation error.
\newblock {\em Journal of Computational and Graphical Statistics}, 30:324--334,
  2021.

\bibitem{Rosenthal1995}
Jeffrey~S. Rosenthal.
\newblock Minorization conditions and convergence rates for {M}arkov chain
  {M}onte {C}arlo.
\newblock {\em Journal of the American Statistical Association}, 90:558--566,
  1995.

\bibitem{Shephard1997}
Neil Shephard and Michael~K. Pitt.
\newblock {Likelihood analysis of non-Gaussian measurement time series}.
\newblock {\em Biometrika}, 84(3):653--667, 1997.

\bibitem{Smith1996}
Richard~L Smith and Luke Tierney.
\newblock Exact transition probabilities for the independence {M}etropolis
  sampler.
\newblock {\em Technical Report}, 1996.

\bibitem{Sur2019}
Pragya Sur and Emmanuel~J. Cand{\`e}s.
\newblock A modern maximum-likelihood theory for high-dimensional logistic
  regression.
\newblock {\em Proceedings of the National Academy of Sciences},
  116:14516--14525, 2019.

\bibitem{Tierney1994}
Luke Tierney.
\newblock Markov chains for exploring posterior distributions.
\newblock {\em The Annals of Statistics}, 22:1701--1728, 1994.

\bibitem{Vats2019}
Dootika Vats, James~M. Flegal, and Galin~L. Jones.
\newblock Multivariate output analysis for {M}arkov chain {M}onte {C}arlo.
\newblock {\em Biometrika}, 106:321--337, 2019.

\bibitem{Villani2003}
C{\'e}dric Villani.
\newblock {\em Topics in Optimal Transportation}.
\newblock Graduate studies in mathematics. American Mathematical Society, 2003.

\bibitem{Villani2008}
C{\'e}dric Villani.
\newblock {\em Optimal Transport: Old and New}.
\newblock 338. Springer-Verlag Berlin Heidelberg, 1 edition, 2009.

\bibitem{Wang2022}
Guanyang Wang.
\newblock Exact convergence rate analysis of the independent
  {M}etropolis-{H}astings algorithms.
\newblock {\em Bernoulli}, 28:2012--2033, 2022.

\bibitem{yang:etal:2016}
Yun Yang, Martin~J. Wainwright, and Michael~I. Jordan.
\newblock On the computational complexity of high-dimensional {B}ayesian
  variable selection.
\newblock {\em Annals of Statistics}, 44:2497--2532, 2016.

\end{thebibliography}

\appendix

\section{Proof of Theorem~\ref{thm:sharp_rate}}
\label{proof:thm:sharp_rate}

The proof will proceed by establishing the upper and lower bounds separately in Lemmas~\ref{lem:MHIS_UB} and~\ref{lem:MHIS_LB}, respectively. This is done largely because the conditions for the upper bound are weaker than those for the lower bound. 

The following definitions will be used in the proofs of Lemmas~\ref{lem:MHIS_UB} and~\ref{lem:MHIS_LB}. First, for $\theta \in \Theta$, real-valued measurable functions $f$, and a Markov kernel $K$, we will use the notation $K^{t} f(\theta) = \int f dK^t(\theta, \cdot) = \int f(\theta') K^t(\theta, d\theta')$ and $K^{0} f(\theta) = f(\theta)$.  Second, recall that for functions $\phi : \R^d \to \R$,
\[
\norm{\phi}_{\text{Lip}(\rho)} = \sup_{x, y, x \not= y} \{ |\phi(x) - \phi(y)| / \rho(x, y)\}.
\]

\begin{lemma}
\label{lem:MHIS_UB}
Let $\e^* = \inf_{\theta \in \Theta} \{ q(\theta)/\pi(\theta) \}$. 
Then
\begin{align*}
\sup_{\theta \in \Theta} \W_\rho(P^t(\theta, \cdot), \Pi)
\le (1 - \e^*)^t \sup_{\theta \in \Theta} \int \rho(\theta, \cdot) d\Pi.
\end{align*}
\end{lemma}

\begin{proof}
Let $\theta \in \Theta$ and let $\phi$ satisfy $\norm{\phi}_{\text{Lip}(\rho)} \le 1$.
The existence of $\e^*$ implies the minorization condition 
$ P(\theta, \cdot) \ge \e^* \Pi(\cdot)$ \cite[][Corollary 4]{Tierney1994} which, in turn, ensures the residual kernel  $R(\theta, \cdot) =  [ P(\theta, \cdot) - \e^* \Pi(\cdot)] / (1 - \e^*)$ is a Markov kernel with invariant distribution $\Pi$.
It then follows that
\begin{align*}
\int \phi dP^t(\theta, \cdot) - \int \phi d\Pi
&= (1 - \e^*) \left[ \int R\phi dP^{t - 1}(\theta, \cdot) - \int \phi d\Pi \right]
\\
&= (1 - \e^*) \left[ \int R\phi dP^{t - 1}(\theta, \cdot) - \int R\phi d\Pi \right]
\\
&\cdots
\\
&= (1 - \e^*)^t \left[ \int \phi dR^t(\theta, \cdot) - \int \phi d\Pi \right].
\end{align*}
Since $\phi$ is Lipschitz with respect to $\rho$, we then have
\begin{align*}
\left| \int \phi dR^t(\theta, \cdot) - \int \phi d\Pi \right|
&= \left| \int \int \left[ \phi(\theta') - \phi(\w) \right] d\Pi(\w) dR^t(\theta, \theta') \right|
\\
&\le \int \int \rho(\theta', \w) d\Pi(\w) dR^t(\theta, \theta')
\\
&\le \sup_{\theta' \in \Theta} \int \rho(\theta', \cdot) d\Pi.
\end{align*}
Taking the supremum with respect to $\phi$ and using the Kantorovich-Rubinstein theorem \citep[Theorem 1.14]{Villani2003},
\begin{align*}
\sup_{\theta \in \Theta} \W_\rho\left( P^t(\theta, \cdot), \Pi \right)
&= \sup_{\theta \in \Theta}
\sup_{\norm{\phi}_{\text{Lip}(\rho)} \le 1} \left[ \int \phi dP^t(\theta, \cdot) - \int \phi d\Pi \right]
\\
&\le (1 - \e^*)^t \sup_{\theta \in \Theta} \int \rho(\theta, \cdot) d\Pi.
\end{align*}
\end{proof}

We now turn our attention to establishing the lower bound. 

\begin{lemma}
\label{lem:MHIS_LB}
Let $\e^* = \inf_{\theta \in \Theta} \{ q(\theta)/\pi(\theta) \}$. Suppose $q$ is lower semicontinuous and $\pi$ is upper semicontinuous on $\Theta$.
Suppose $\Theta$ can be expressed as a countable union of compact sets.
If $\rho(\cdot, \cdot) \le 1$, then
\begin{align*}
\sup_{\theta \in \Theta} \W_\rho(P^t(\theta, \cdot), \Pi) \ge (1 - \e^*)^t \inf_{\theta \in \Theta} \int \rho(\cdot, \theta) d\Pi .
\end{align*}
\end{lemma} 

\begin{proof}
Since $\Theta$ can be expressed as a countable union of compact sets, there is a sequence of compact sets $B_n \subseteq B_{n + 1} \subseteq \Theta$ increasing to $\Theta = \cup_{n = 1}^{\infty} B_n$.
We can assume $\Pi(B_n) > 0$ or else we can take $n$ large enough so this holds.
Since $\pi, q > 0$ and $\pi$ is upper semicontinuous on $\Theta$, then $q/\pi$ is lower semicontinuous on $\Theta$.
By Lemma~\ref{lemma:seqinf}, we have that $\inf_{\theta \in B_n} \{ q(\theta)/\pi(\theta) \}$ is monotonically non-increasing to $\e^* = \inf_{\theta \in \Theta} \{ q(\theta)/\pi(\theta) \}$.
Since we have assumed lower semicontinuity, the $\inf_{\theta \in K} \{ q(\theta)/\pi(\theta) \}$ is attained over any compact set $K \subseteq \Theta$.
Then define the sequence 
\begin{align}
\theta^*_n = \text{argmin}_{\theta \in B_n} \{ q(\theta)/\pi(\theta) \}. \label{eq:theta_star_n}
\end{align}
We can then define the sequence 
\[
\e_{\theta^*_n} = \inf_{\theta \in B_n} \{ q(\theta)/\pi(\theta) \} = q(\theta^*_n)/\pi(\theta^*_n)
\] 
and this is monotonically non-increasing to $\e^*$.

Define $P_n$ to be the Metropolis-Hastings independence kernel with independent proposal $Q$ with density $q$ and target distribution $\Pi(\cdot | B_n)$ with density $\pi(\cdot | B_n) = \pi(\cdot) I_{B_n}(\cdot) / \Pi(B_n)$.
By construction, $\Pi(B_n) > 0$ and this is well-defined.
The key part of the proof is that if we start at any $\theta_n \in B_n$, this kernel $P_n$ and the kernel $P$ only disagree outside of $B_n$.
For $\theta_n \in B_n$, we have $\pi(\theta_n) > 0$, $I_{B_n}(\theta_n) = 1$, and since $\Theta \equiv \text{supp}(q)$ by assumption, then $q(\theta_n) > 0$. Also, if $y \in B_n^c \cap \Theta$, then $\min\left\{ \frac{\pi(y)I_{B_n}(y) q(\theta_n)}{\pi(\theta_n) q(y)}, 1 \right\} = 0$.
Let $M_1(\R^d)$ be the set of measurable functions $\phi : \R^d \to \R$ with $\sup_{x \in \R^d} |\phi(x)| \le 1$.
Therefore, for any $\theta_n \in B_n$ and any function $\phi \in M_1(\R^d)$,
\begin{align*}
&\int_{\R^d} \phi dP_n(\theta_n, \cdot)
= \int_{B_n} \phi(y) \min\left\{ \frac{\pi(y) q(\theta_n)}{\pi(\theta_n) q(y)}, 1 \right\} q(y) dy
\\
&\hspace{.4cm}+ \phi(\theta_n) \left( 1 - \int_{B_n} \min\left\{ \frac{\pi(y) q(\theta_n)}{\pi(\theta_n) q(y)}, 1 \right\} q(y) dy \right).
\end{align*}

Let $\e \in (0, 1 - \e^*)$.
Since $Q$ and $\Pi$ are probability measures, we may then choose $n_{\e}$ sufficiently large such that for all $n \ge n_{\e}$,
\[
2 \max\left\{ \Pi(B_n^c), Q(B_n^c) \right\}
\le \e/2.
\]
We then have
\begin{align}
&\sup_{\theta_n \in B_n} \sup_{\phi \in M_1(\R^d)}  \left| \int_{\R^d} \phi dP_n(\theta_n, \cdot) - \int_{\R^d} \phi dP(\theta_n, \cdot) \right| \nonumber
\\
&= \sup_{\theta_n \in B_n} \sup_{\phi \in M_1(\R^d)}  \Bigg| \int_{B_n^c \cap \Theta} \phi(y) \min\left\{ \frac{\pi(y) q(\theta_n)}{\pi(\theta_n) q(y)}, 1 \right\} q(y) dy \nonumber
\\
&\hspace{3.3cm}+ \phi(\theta_n) \int_{B_n^c \cap \Theta} \min\left\{ \frac{\pi(y) q(\theta_n)}{\pi(\theta_n) q(y)}, 1 \right\} q(y) dy \Bigg| \nonumber
\\
&\le 2 \int_{B_n^c} q(y) dy \nonumber
\\
&\le \e/2. \label{eq:conv_result1}
\end{align}
Similarly,
\begin{align}
&\sup_{\phi \in M_1(\R^d)} \left| \int_{\R^d} \phi d\Pi(\cdot | B_n) - \int_{\R^d} \phi d\Pi \right| \nonumber
\\
&= \sup_{\phi \in M_1(\R^d)} \left| \int_{\R^d} \phi \left( 1 - \Pi(B_n) \right) d\Pi(\cdot | B_n) - \int_{\R^d} \phi d\Pi(\cdot | B_n^c) \Pi(B_n^c) \right| \nonumber
\\
&= \Pi(B_n^c) \sup_{\phi \in M_1(\R^d)} \left| \int_{\R^d} \phi d\Pi(\cdot | B_n) - \int_{\R^d} \phi d\Pi(\cdot | B_n^c) \right| \nonumber
\\
&\le 2 \Pi(B_n^c) \nonumber
\\
&\le \e/2. \label{eq:conv_result2}
\end{align}

With $\theta^*_n$ as in \eqref{eq:theta_star_n}, let $\psi_n(\cdot) = -\rho(\cdot, \theta_n^*)$. 
Then for any $x, y \in \R^d$, 
\begin{align}
| \psi_n(x) - \psi_n(y) | \le \rho(x, y)
\label{eq:rho_lipschitz}
\end{align} 
and $\psi_n \in M_1(\R^d)$.
%Then this is a $\rho$-Lipschitz function and $\psi_n \in M_1(\R^d)$.
Since $\Pi$ is invariant for the kernel $P$, 
\begin{align}
\int_{\R^d} \psi_n dP^t(\theta^*_{n}, \cdot) - \int_{\R^d} \psi_n d\Pi
&= \int_{\R^d} P^{t-1} \psi_n(\cdot) dP(\theta^*_{n}, \cdot) - \int_{\R^d} P^{t-1} \psi_n(\cdot) d\Pi(\cdot). \label{eq:P_identity}
\end{align}
Now for any integer $s$ with $1 \le s \le t$, the function $P^{s} \psi_n \in M_1(\R^d)$ since $P$ is a Markov kernel. 
Since $\theta^*_n \in B_n$ and $\pi(\theta^*_n) > 0$, using \eqref{eq:conv_result1}, \eqref{eq:conv_result2}, and \eqref{eq:P_identity},
\begin{align}
&\int_{\R^d} \psi_n dP^t(\theta^*_{n}, \cdot) - \int_{\R^d} \psi_n d\Pi \notag
\\
&\ge \int_{\R^d} P^{t-1} \psi_n dP_{n}(\theta^*_{n}, \cdot) - \int_{\R^d} P^{t-1} \psi_n d\Pi(\cdot | B_n) - \e. \label{eq:lb_eps}
\end{align}
By construction of $\theta^*_n$ in \eqref{eq:theta_star_n}, we have 
\begin{align*}
\inf_{\theta \in B_n} \{ q(\theta) / \pi(\theta | B_n) \}
&= \Pi(B_n) \inf_{\theta \in B_n} \{ q(\theta) / \pi(\theta) \}
\\
&= \Pi(B_n) q(\theta_n^*) / \pi(\theta_n^*) 
\\
&= \e_{\theta^*_n} \Pi(B_n)
\\
&= q(\theta_n^*) / \pi(\theta_n^* | B_n).
\end{align*}
For measurable $A \subset \R^d$ \citep[Remark 1, Theorem 2]{Wang2022}, we then have the identity
\begin{align}
P_{n}(\theta^*_{n}, A)
= \e_{\theta^*_n} \Pi(B_n) \Pi(A|B_n)
+ \left( 1 - \e_{\theta^*_n} \Pi(B_n) \right)\delta_{\theta^*_{n}}(A). \label{eq:P_n_identity}
\end{align}
%This 
Since $P^{t-1} \psi_n$ is a bounded measurable function, \eqref{eq:P_n_identity} gives the identity:
\begin{align}
&\int_{\R^d} P^{t-1} \psi_n(\cdot) dP_{n}(\theta^*_{n}, \cdot) - \int_{\R^d} P^{t-1} \psi_n(\cdot) d\Pi(\cdot | B_n) \notag
\\
&= (1 - \e_{\theta^*_n} \Pi(B_n)) \left( P^{t-1} \psi_n(\theta^*_n) - \int_{\R^d} P^{t-1} \psi_n(\cdot) d\Pi(\cdot | B_n) \right). \label{eq:identityexact}
\end{align}
Using \eqref{eq:lb_eps} in the first inequality, \eqref{eq:identityexact} in the second inequality, \eqref{eq:conv_result2} in the third inequality, and using the invariance of $\Pi$ for the Markov kernel $P$ in the last inequality,
\begin{align*}
&\int_{\R^d} P^{t-1} \psi_n P(\theta^*_{n}, \cdot) - \int_{\R^d} P^{t-1} \psi_n d\Pi
\\
&\ge \int_{\R^d} P^{t-1} \psi_n dP_{n}(\theta^*_{n}, \cdot) - \int_{\R^d} P^{t-1} \psi_n d\Pi(\cdot | B_n) - \e.
\\
&\ge (1 - \e_{\theta^*_n} \Pi(B_n)) \left( P^{t-1} \psi_n(\theta^*_n) - \int_{\R^d} P^{t-1} \psi_n d\Pi(\cdot | B_n) \right) - \e
\\
&\ge (1 - \e_{\theta^*_n} \Pi(B_n)) \left( P^{t-1} \psi_n(\theta^*_n) - \int_{\R^d} P^{t-1} \psi_n d\Pi \right) - 2 \e
\\
&\ge (1 - \e_{\theta^*_n} \Pi(B_n)) \left( \int_{\R^d} P^{t-2} \psi_n dP(\theta^*_{n}, \cdot) - \int_{\R^d} P^{t-2} \psi_n d\Pi \right) - 2\e.
\end{align*}
Applying this inequality recursively and using the definition of $\psi_n$
\begin{align}
&\int_{\R^d} \psi_n dP^t(\theta^*_{n}, \cdot) - \int_{\R^d} \psi_n d\Pi \nonumber
\\
&= \int_{\R^d} P^{t-1} \psi_n dP(\theta^*_{n}, \cdot) - \int_{\R^d} P^{t-1} \psi_n d\Pi \nonumber
\\
&\ge (1 - \e_{\theta^*_n} \Pi(B_n))^t \left( \psi_n(\theta^*_n) - \int_{\R^d} \psi_n d\Pi \right) - 2 \e \sum_{s = 0}^{t-1} (1 - \e_{\theta^*_n} \Pi(B_n))^s \nonumber
\\
&= (1 - \e_{\theta^*_n} \Pi(B_n))^t \int_{\R^d} \rho(\theta, \theta_n^*) d\Pi - 2 \e \sum_{s = 0}^{t-1} (1 - \e_{\theta^*_n} \Pi(B_n))^s. \label{eq:recursive_wass_lb}
\end{align}
Since $\Pi(B_n) \to 1$ and $\e_{\theta^*_n} \to \e^*$, we may take $n$ large enough so that
\[
\left| \e_{\theta^*_n} \Pi(B_n) - \e^* \right| \le \e.
\]
For all large enough $n$ and since $\e < 1 - \e^*$, we lower bound \eqref{eq:recursive_wass_lb} to get
\begin{align}
&\int_{\R^d} \psi_n dP^t(\theta^*_{n}, \cdot) - \int_{\R^d} \psi_n d\Pi
\nonumber
\\
&\ge (1 - \e^* - \e)^t \inf_{\theta \in \Theta} \int \rho(\cdot, \theta) d\Pi - 2 \e \sum_{s = 0}^{t-1} (1 - \e^* + \e)^s. \label{eq:final_lb}
\end{align}
Combining \eqref{eq:rho_lipschitz} and \eqref{eq:final_lb}, we lower bound the Wasserstein distance with
\begin{align*}
&\sup_{\theta \in \Theta} \W_\rho(P^t(\theta, \cdot), \Pi)
\\
&\ge \W_\rho(P^t(\theta^*_{n}, \cdot), \Pi)
\\
&\ge (1 - \e^* - \e)^t \inf_{\theta \in \Theta} \int \rho(\cdot, \theta) d\Pi - 2 \e \sum_{s = 0}^{t-1} (1 - \e^* + \e)^s.
\end{align*}
Since this holds for all small $\e$, the proof is complete by taking the limit as $\e \downarrow 0$.
\end{proof}

\section{Technical Lemmas}

%However, we begin with a preliminary technical result we will use in the proof of Lemma~\ref{lem:MHIS_LB}.

\begin{lemma}
\label{lemma:seqinf}
Let $\e^* = \inf_{\theta \in \Theta} \{ q(\theta)/\pi(\theta) \}$.
Suppose there is a sequence of compact sets $B_n \subseteq B_{n + 1} \subseteq \Theta$ increasing to $\Theta = \cup_{n = 1}^{\infty} B_n$.
Define the sequence $\e_n = \inf_{\theta \in B_n} q(\theta)/\pi(\theta)$.
Then $\e_n$ is monotonically non-increasing to its limit $\e^*$.
\end{lemma}

\begin{proof}
By definition of the infimum, $\e_n \ge \e_{n + 1}$ and $\e_n \ge \e^*$. Hence the sequence $\e_n$ converges.
Let $\delta \in (0, \infty)$. By the definition of the infimum, we can choose $\theta_\delta \in \Theta$ with $\pi(\theta_\delta) > 0$ such that
\[
q(\theta_\delta)/\pi(\theta_\delta) - \delta \le \e^*.
\]
We can choose $B_{n_\delta}$ such that $\theta_\delta \in B_{n_\delta}$.
Then
\[
\e_{n_\delta} - \delta
\le q(\theta_\delta)/\pi(\theta_\delta) - \delta
\le \e^*.
\]
It follows that for any $n \ge n_\delta$
\[
| \e_{n} - \e^* |
= \e_{n} - \e^*
\le \e_{n_\delta} - \e^*
\le \delta.
\]
Therefore, $\lim_n \e_{n} = \e^*$.
\end{proof}

\begin{lemma}
\label{lemma:subgradient_inequality}
Let $C \in \R^{d \times d}$ be a positive-definite, symmetric matrix and $\alpha \in (0, \infty)$.
Let $f : \R^d \to \R$ and suppose $\theta \mapsto f(\theta) - \alpha \theta^T C^{-1} \theta / 2$ is convex for all points on $\R^d$.
Then there exists $\theta^* = \text{argmin}_{\theta \in \R^d} f(\theta)$ and
\[
f(\theta)
\ge f(\theta^*) + \frac{\alpha}{2} \left( \theta - \theta^* \right)^T C^{-1} (\theta - \theta^*).
\]
\end{lemma}

\begin{proof}
Since the function
  $f(\theta) - \alpha \theta^T C^{-1} \theta/2$ is convex for
  all points on $\R^d$, it follows that for any $\lambda \in [0, 1]$ and any $(\theta, \theta') \in \R^d \times \R^d$,
\begin{align*}
f(\lambda \theta + (1 - \lambda) \theta')
\le \lambda f(\theta) + (1 - \lambda) f(\theta') - \frac{\alpha}{2} \lambda (1 -
  \lambda) \left( \theta' - \theta \right)^T C^{-1} (\theta' - \theta). 
\end{align*}
Since
$C^{-1}$ is positive-definite, then
$\alpha \lambda (1 - \lambda) \left( \theta' - \theta
\right)^T C^{-1} (\theta' - \theta)/2$ is nonnegative and this implies
that $f$ is a convex function.
It can also be shown that $\lim_{\norm{\theta} \to +\infty} f(\theta) = +\infty$ and since $f$ is lower semicontinuous, then $f$ attains its minimum $\theta^* \in \R^d$.
The right directional derivative
\[
f'(\theta^*; \theta) = \lim_{t \downarrow 0} t^{-1} \left[
  f(\theta^* + t \theta) - f(\theta^*) \right]
\]
exists for all points
$\theta \in \R^d$ \cite[][Theorem 3.1.12]{Nesterov2018}.  For
$\lambda \in (0, 1)$, we have
\begin{align*}
&\frac{1}{ (1 - \lambda)} \frac{1}{\lambda} \left[ f(\theta^* + \lambda \left( \theta - \theta^* \right)) - f(\theta^*) \right]
- \frac{1}{(1 - \lambda)} \left( f(\theta) - f(\theta^*) \right)
\\
&\le - \frac{\alpha}{2} \left( \theta - \theta^* \right)^T C^{-1} (\theta - \theta^*).
\end{align*}
Taking the limit with $\lambda \downarrow 0$, we have that
\[
f'(\theta^*; \theta - \theta^*)
- f(\theta) + f(\theta^*)
\le - \frac{\alpha}{2} \left( \theta - \theta^* \right)^T C^{-1} (\theta - \theta^*).
\]
Since $\theta^*$ is the minimum of $f$, then the right directional derivative satisfies $f'(\theta^*; \theta - \theta^*) \ge 0$ for all $\theta \in \R^d$.
Therefore for all $\theta \in \R^d$,
\[
f(\theta)
\ge f(\theta^*) + \frac{\alpha}{2} \left( \theta - \theta^* \right)^T C^{-1} (\theta - \theta^*).
\]
\end{proof}

\end{document}